\documentclass[11pt]{elsarticle}

\usepackage{lineno,hyperref}
\modulolinenumbers[5]
\usepackage[all]{xy}
\usepackage{tabularx}

\usepackage{amsmath,amssymb}
\usepackage{amsthm}
\usepackage{latexsym}
\usepackage{cleveref}
\usepackage{enumerate}
\usepackage{bm}
\usepackage{multirow}
\usepackage{geometry}
\usepackage{indentfirst}
\usepackage{fancyhdr}
\usepackage{graphicx}
\usepackage{subfigure}
\usepackage[noend]{algpseudocode}
\usepackage{algorithmicx,algorithm}
\usepackage{threeparttable}
\usepackage{booktabs}
\usepackage{makecell}
\usepackage{color}
\usepackage{float}
\usepackage{url}

\newtheoremstyle{mystyle}
{3pt}
{3pt}
{\rm}
{}
{\bfseries}
{.}
{.5em}
{}
\newcommand{\RNum}[1]{\uppercase\expandafter{\romannumeral #1\relax}}
\newtheorem{theorem}{Theorem}[section]
\newtheorem{proposition}[theorem]{Proposition}
\newtheorem{lemma}[theorem]{Lemma}
\newtheorem{corollary}[theorem]{Corollary}

\newtheorem{remark}[theorem]{Remark}
\newtheorem{example}[theorem]{Example}

\date{}

\begin{document}
\begin{frontmatter}
\title{Endomorphism Rings of Supersingular Elliptic Curves over $\mathbb{F}_p$ and Binary Quadratic Forms}

\author[NUDT]{Guanju Xiao}
\ead{gjXiao@amss.ac.cn}

\author[NUDT]{Zijian Zhou}
\ead{zhouzijian122006@163.com}

\author[CAS,UCAS]{Yingpu Deng}
\ead{dengyp@amss.ac.cn}

\author[NUDT]{Longjiang Qu\corref{mycorrespondingauthor}}
\cortext[mycorrespondingauthor]{Corresponding author}
\ead{ljqu_happy@hotmail.com}

\address[NUDT]{College of Science, National University of Defense Technology\\
Changsha 410073, China}

\address[CAS]{Key Laboratory of Mathematics Mechanization, NCMIS, Academy of Mathematics and Systems Science, Chinese Academy of Sciences, Beijing 100190, China}
\address[UCAS]{School of Mathematical Sciences, University of Chinese Academy of Sciences\\
Beijing 100049, China  }

\begin{abstract}
It is well known that there is a one-to-one correspondence between supersingular $j$-invariants up to the action of $\text{Gal}(\mathbb{F}_{p^2}/\mathbb{F}_p)$ and type classes of maximal orders in $B_{p,\infty}$ by Deuring's theorem. Interestingly, we establish a one-to-one correspondence between $\mathbb{F}_p$-isomorphism classes of supersingular elliptic curves and primitive reduced binary quadratic forms with discriminant $-p$ or $-16p$. Due to this correspondence and the fact that $\mathbb{F}_p$-isogenies between elliptic curves could be represented by quadratic forms, we show that actions of these isogenies on supersingular elliptic curves over $\mathbb{F}_p$ are compatible with the composition of quadratic forms. Based on these results, we reduce the security of CSIDH cryptosystem to computing this correspondence explicitly.
\end{abstract}

\begin{keyword}
Supersingular Elliptic Curve, Endomorphism Ring, Binary Quadratic Form
\end{keyword}
\end{frontmatter}
\section{Introduction}

In 1941, Deuring \cite{MR5125} proved that there is a one-to-one correspondence between type classes of maximal orders in $B_{p,\infty}$ and isomorphism classes of supersingular elliptic curves up to the action of $\text{Gal}(\mathbb{F}_{p^2}/\mathbb{F}_p)$, where $B_{p,\infty}$ is the unique quaternion algebra over $\mathbb{Q}$ ramified only at $p$ and $\infty$.

For a supersingular elliptic curve $E$ defined over $\mathbb{F}_p$, its endomorphism ring $\text{End}(E)$ contains a root of $x^2+p=0$, which means that the imaginary quadratic order $\mathbb{Z}[\sqrt{-p}]$ can be embedded into $\text{End}(E)$ (see \cite{MR3451433}). Let $q \equiv 3 \pmod 8$ be a prime with $\left( \frac{p}{q} \right)=-1$. Ibukiyama \cite{MR683249} gave two kinds of maximal orders $\mathcal{O}(q,r)$ and $\mathcal{O}'(q,r')$ in $B_{p,\infty}$ where $r$ (resp. $r'$) satisfies $r^2 +p \equiv 0 \pmod q$ (resp. $(r')^2 +p \equiv 0 \pmod {4q}$), and proved that they are isomorphic to endomorphism rings of some supersingular elliptic curves which are defined over $\mathbb{F}_p$. We assume $p>3$ in this paper, because there exists only one supersingular $j$-invariant defined over $\overline{\mathbb{F}}_2$ or $\overline{\mathbb{F}}_3$.

We study the structure of maximal orders $\mathcal{O}(q,r)$'s and $\mathcal{O}'(q,r')$'s. As we see, the imaginary quadratic order $\mathbb{Z}[(1+\sqrt{-p})/{2}]$ can be embedded into every maximal order $\mathcal{O}'(q,r')$. Moreover, we can reduce the maximal order $\mathcal{O}'(q,r')$ to a reduced binary quadratic form with discriminant $-p$ which can represent the prime $q$. On the contrary, given a reduced form $(a,b,c)$ with discriminant $-p$, we can get a unique supersingular $j$-invariant by computing the greatest common divisor of the Hilbert class polynomials $H_{-4a}(X)$ and $H_{-4c}(X)$ in $\mathbb{F}_p$. Similar results also hold for $\mathcal{O}(q,r)$'s, while the corresponding reduced forms are of discriminant $-16p$ which are not in the principal genus. In particular, we establish a correspondence between maximal orders $\mathcal{O}'(q,r')$'s (resp. $\mathcal{O}(q,r)$'s) and reduced binary quadratic forms with discriminant $-p$ (resp. $-16p$).

Supersingular elliptic curves have already been shown to be in bijection to ternary quadratic forms of discriminant $-p$. That correspondence, called the Brandt--Sohn correspondence \cite{Brandt, Sohn}, also goes through the intermediary of maximal orders of quaternions and $\text{End}(E)$ by Deuring. Our results are interesting, since binary quadratic forms are simpler than ternary quadratic forms.

We also study the connections of this correspondence and isogenies between supersingular elliptic curves defined over $\mathbb{F}_p$. Let $E_1$ and $E_2$ be two supersingular elliptic curves defined over $\mathbb{F}_p$. There exists an isogeny $\phi:E_1 \to E_2$ defined over $\mathbb{F}_p$. As we know, the isogeny $\phi$ can be represented by an ideal in $\mathbb{Z}[\sqrt{-p}]$ or $\mathbb{Z}[(1+\sqrt{-p})/{2}]$. Equivalently, we can associate to this ideal a binary quadratic form $g$. If the curve $E_1$ corresponds to the quadratic form $f$, then we prove that the curve $E_2$ corresponds to the form $fg^2$. This shows an explicit relation between isogenies of supersingular elliptic curves defined over $\mathbb{F}_p$ and their endomorphism rings.

Note that this correspondence may have an application in isogeny-based cryptography. Isogeny-based cryptography is a relatively new kind of elliptic-curve cryptography, whose security relies on the problem of finding an explicit isogeny between two given isogenous elliptic curves over a finite field. One of the isogeny-based cryptographic protocols is CSIDH proposed by Castryck et al. \cite{MR3897883} in 2018. CSIDH uses the action of an ideal class group on the set of supersingular elliptic curves defined over $\mathbb{F}_p$. In this paper, we reduce the security of CSIDH to computing the correspondence explicitly.

The remainder of this paper is organized as follows. In Section 2, we review some preliminaries on supersingular elliptic curves defined over $\mathbb{F}_p$ and binary quadratic forms. We prove the correspondence between supersingular elliptic curves over $\mathbb{F}_p$ and binary quadratic forms in Section 3. In Section 4, we prove that the composition of forms corresponds to the action of isogenies on elliptic curves. In Section 5, we discuss the applications of these results to CSIDH. Finally, we make a conclusion in Section 6.

\section{Preliminaries}

\subsection{Elliptic curves and isogenies}
We present some basic facts about elliptic curves over finite fields, and the readers can refer to \cite{MR2514094} for more details.

Let $\mathbb{F}_q$ be a finite field with characteristic $p>3$. An elliptic curve $E$ defined over $\mathbb{F}_q$ can be written as a Weierstrass model $E:Y^2=X^3+aX+b$ where $a,b \in \mathbb{F}_q$ and $4a^3+27b^2 \neq 0$. The $j$-invariant of $E$ is $j(E)=1728\cdot 4a^3/(4a^3+27b^2)$. Different elliptic curves with the same $j$-invariant are isomorphic over the algebraic closure of the base field. The chord-and-tangent addition law makes of $E(\mathbb{F}_q)=\left\{(x,y)\in \mathbb{F}_q ^2:y^2=x^3+ax+b\right\} \cup \left\{\infty \right\}$ an abelian group, where $\infty$ is the point at infinity.

Let $E_1$ and $E_2$ be elliptic curves defined over $\mathbb{F}_q$. An isogeny from $E_1$ to $E_2$ is a morphism $\phi:E_1 \rightarrow E_2$ satisfying $\phi(\infty)=\infty$. We always assume $\phi \neq 0$. The isogeny $\phi$ is a surjective group homomorphism with finite kernel, and it is called a $\mathbb{F}_q$-isogeny if it is defined over $\mathbb{F}_q$. Every $\mathbb{F}_q$-isogeny can be represented as $\phi = (r_1(X), r_2(X) \cdot Y)$, where $r_1(X), r_2(X) \in \mathbb{F}_q(X)$. If $r_1(X)=p_1(X)/q_1(X)$ with $p_1, q_1 \in \mathbb{F}_q[X]$ and $\text{gcd}(p_1,q_1)=1$, then the degree of $\phi$ is $\text{max}(\text{deg}\ p_1, \text{deg}\ q_1)$. The isogeny $\phi$ is said to be separable if $r_1 '(X) \neq 0$. Note that all isogenies of prime degree $\ell\neq p$ are separable.

An endomorphism of $E$ is an isogeny from $E$ to itself or the zero map. The Frobenius map $\pi :(x,y) \mapsto (x^q,y^q)$ is an inseparable endomorphism. The characteristic polynomial of $\pi$ is $x^2-tx+q$, where $t$ is the trace of $\pi$. It is well known that $E$ is supersingular (resp. ordinary) if $p\mid t$ (resp. $p \nmid t$). Moreover, the $j$-invariant of every supersingular elliptic curve over $\overline{\mathbb{F}}_p$ is proved to be in $\mathbb{F}_{p^2}$ and it is called a supersingular $j$-invariant.

The set $\text{End}(E)$ of all endomorphisms of $E$ over $\overline{\mathbb{F}}_q$ form a ring under the usual addition and composition as multiplication. The set $\text{End}_{\mathbb{F}_q}(E)$ is a subring of $\text{End}(E)$ which contains all the endomorphisms over $\mathbb{F}_q$.

\subsection{Imaginary quadratic orders and binary quadratic forms}
We recall some basic facts about imaginary quadratic orders and binary quadratic forms. The general references are \cite{MR1012948, MR3236783}.

Let $K$ be an imaginary quadratic field. Denote its ring of integers by $O_K$ and its discriminant by $D_K$. Let $O$ be an order of $K$. The conductor of $O$ is $f=\left[O_{K}: O\right]$, and the discriminant of $O$ is $f^{2} D_{K}$.

Let $I(O)$ be the group of proper fractional $O$-ideals. The principal proper fractional $O$-ideals give a subgroup $P(O)\subset I(O)$. The ideal class group of $O$ is $C(O)=I(O) / P(O)$ and the class number of $O$ is $h(O)=\# C(O)$. Let $I(f)$ be the group of fractional $O_{K}$-ideals prime to $f$. Let $P_{\mathbb{Z}}(f)$ be the subgroup of $I(f)$ generated by principal ideals of the form $\alpha O_K$, where $\alpha \in O_K$ satisfies $\alpha \equiv a \bmod f O_{K}$ for $a \in \left( \mathbb{Z}/f\mathbb{Z} \right) ^{\times}$. The group $C(O)$ is canonically isomorphic to the ring class group $I(f) / P_{\mathbb{Z}}(f)$. We often denote the class number of $O$ by $h(D)$ if the discriminant of $O$ is $D$.

We can relate the class numbers $h(m^2D)$ and $h(D)$ as follows.
\begin{proposition}(\cite[Corollary 7.28]{MR3236783})
Let $D \equiv 0,1 \pmod 4$ be a negative integer and $m$ a positive integer. Then
$$h(m^2D)=\frac{h(D)m}{[O^{\times}:O'^{\times}]} \prod_{p | m} \left( 1- \left( \frac{D}{p} \right) \frac{1}{p} \right),$$
where $O$ and $O'$ are the orders of discriminant $D$ and $m^2D$ respectively.
\end{proposition}

We denote a quadratic form $f(X,Y)=aX^2+bXY+cY^2$ by $(a,b,c)$ with $a,b,c \in \mathbb{Z}$. Its discriminant is $D=b^2-4ac$. Note that all forms are positive definite in this paper, which means $D < 0$ and $a>0$. A form $(a,b,c)$ is called primitive if $\text{gcd}(a,b,c)=1$. Two forms $f(X,Y)$ and $g(X,Y)$ are equivalent if there are integers $m$, $n$, $r$ and $s$ such that
$$f(X,Y)=g(mX+nY,rX+sY) \quad \text{and} \quad ms-nr=1. $$
In this case, we write $f\sim g$.
 If a form $(a,b,c)$ satisfies
\begin{itemize}
  \item [(1)] $-a < b \le a$,
  \item [(2)] $ a\le c $,
  \item [(3)] if $a=c$ then $b \ge 0$,
\end{itemize}
then it is reduced. Every primitive positive definite form is equivalent to a unique reduced form (see \cite[Theorem 2.8]{MR3236783}).

For every $D=b^2-4ac$, there is an imaginary quadratic order $O$ such that the discriminant of $O$ is $D$. By \cite[Theorem 6.15]{MR1012948}, every ideal in $O$ has an integral basis, and we use $[a,\beta]=\mathbb{Z}a+\mathbb{Z}\beta$ to represent an ideal in $O$, where $a$ is a rational integer and $\beta$ is a quadratic algebra integer. We can relate the ideal class group $C(O)$ and the form class group $C(D)$ defined in \cite[\S 3]{MR3236783} as follows:
\begin{proposition}\label{t3}(\cite[Theorem 7.7]{MR3236783})
 Let $O$ be the order of discriminant $D$ in an imaginary quadratic field $K$. Then:
\begin{itemize}
  \item [(1)] If $f(X,Y)=aX^2+bXY+cY^2$ is a primitive positive definite quadratic form of discriminant $D$, then $[a, (-b+\sqrt{D})/2]$ is a proper ideal of $O$.
  \item [(2)] The map sending $f(X,Y)$ to $[a, (-b+\sqrt{D})/2]$ induces an isomorphism $\iota : C(D) \longrightarrow C(O)$.
  \item [(3)] A positive integer $m$ is represented by a form $f(X,Y)$ if and only if $m$ is the norm of some ideal $\mathfrak{a}$ in the corresponding ideal class in $C(O)$.
\end{itemize}
\end{proposition}

\begin{remark}
If $\mathfrak{a}=[a,\beta]$ is a proper $O$-ideal with $\operatorname{Im}(\beta / a)>0$, then
$$f(X,Y)=\frac{N(a X-\beta Y)}{N(\mathfrak{a})}$$
is a positive definite form of discriminant $D$. On the level of classes, this map is the inverse to the map $\iota$ in Proposition \ref{t3}.
\end{remark}

By Proposition \ref{t3}, one can use primitive reduced quadratic forms to represent proper ideal classes in an imaginary quadratic order, and the composition of forms is equivalent to the multiplication of ideals.

\subsection{The Isogeny Graphs of Supersingular Elliptic Curves over $\mathbb{F}_p$}

We discuss the isogeny graphs of supersingular elliptic curves over $\mathbb{F}_p$, and the reader can refer to \cite{MR3451433} for more details.

Let $E$ be a supersingular elliptic curve over $\mathbb{F}_p$ with $p>3$. We know that $\text{End}_{\mathbb{F}_p}(E)$ is an order in $K=\mathbb{Q}(\sqrt{-p})$. Furthermore, we have
$$\mathbb{Z}[\sqrt{-p}] \subseteq \operatorname{End}_{\mathbb{F}_p}(E) \subseteq O_K,$$
where $O_K$ is the integer ring of $K$.

\begin{itemize}
  \item If $p \equiv 1 \pmod 4$, we always have $\operatorname{End}_{\mathbb{F}_p}(E)=\mathbb{Z}[\sqrt{-p}]$ for a supersingular elliptic curve $E$ over $\mathbb{F}_p$.
  \item If $p \equiv 3 \pmod 4$, then $\mathbb{Z}[\sqrt{-p}]$ has conductor $2$ in $O_K$ and $\operatorname{End}_{\mathbb{F}_p}(E)$ must be one of these two orders.
\end{itemize}

We say that the curve $E$ is on the surface (resp. floor) if $\operatorname{End}_{\mathbb{F}_{p}} (E)=\mathcal{O}_{K}$ (resp. $\mathbb{Z}[\sqrt{-p}]$ ). Note that surface and floor coincide for $p \equiv 1 \pmod 4$.

Let $\phi$ be an $\ell$-isogeny between supersingular elliptic curves $E$ and $E'$ over $\mathbb{F}_{p}$. If $\operatorname{End}_{\mathbb{F}_{p}} (E) \cong \operatorname{End}_{\mathbb{F}_{p}} (E')$, then $\phi$ is called horizontal. If $E$ is on the floor and $E'$ is on the surface (resp. $E$ on the surface and $E'$ on the floor), $\phi$ is called $\ell$-isogeny up (resp. down).

\begin{proposition}(\cite[Theorem 2.7]{MR3451433})
Let $p>3$ be a prime.

(1) $p \equiv 1 \pmod 4$: There are $h(-4 p)$ $\mathbb{F}_{p}$-isomorphism classes of supersingular elliptic curves over $\mathbb{F}_{p}$, all having the same endomorphism ring $\mathbb{Z}[\sqrt{-p}]$. From every one there is one outgoing $\mathbb{F}_{p}$-rational horizontal $2$-isogeny as well as two horizontal $\ell$-isogenies for every prime $\ell>2$ with $\left(\frac{-p}{\ell}\right)=1$.

(2) $p \equiv 3 \pmod 4$: There are two levels in the supersingular isogeny graph. From each vertex there are two horizontal $\ell$-isogenies for every prime $\ell>2$ with $\left(\frac{-p}{\ell}\right)=1$.
\begin{itemize}
  \item If $p \equiv 7 \pmod 8$, on each level $h(-p)$ vertices are situated. Surface and floor are connected 1:1 with $2$-isogenies and on the surface we also have two horizontal 2-isogenies from each vertex.
  \item If $p \equiv 3 \pmod 8$, there are $h(-p)$ vertices on the surface and $3 h(-p)$ on the floor. Surface and floor are connected 1:3 with $2$-isogenies, and there are no horizontal 2-isogenies.
\end{itemize}
\end{proposition}

\subsection{Endomorphism rings of supersingular elliptic curves}
For a supersingular elliptic curve $E$ over $\mathbb{F}_{p^2}$, the endomorphism ring $\text{End}(E)$ is isomorphic to a maximal order of $B_{p,\infty}$, where $B_{p,\infty}$ is the unique quaternion algebra over $\mathbb{Q}$ ramified only at $p$ and $\infty$.

An order $\mathcal{O}$ of $B_{p,\infty}$ is a subring of $B_{p,\infty}$ which is also a lattice, and it is called a maximal order if it is not properly contained in any other order. Let $\mathcal{O}$ be a maximal order of $B_{p,\infty}$ and $I$ a left ideal of $\mathcal{O}$. Define the left order $\mathcal{O}_L(I)$ and the right order $\mathcal{O}_R(I)$ of $I$ by
$$\mathcal{O}_L(I)=\left\{x\in B_{p,\infty}:xI\subseteq I \right \}, \quad \mathcal{O}_R(I)=\left\{x\in B_{p,\infty}:Ix \subseteq I\right \}  .$$
Moreover, $\mathcal{O}_L(I)=\mathcal{O}$, and $\mathcal{O}_R(I)=\mathcal{O}'$ is also a maximal order, in which case we say that $I$ connects $\mathcal{O}$ and $\mathcal{O}^{'}$. The reduced norm of $I$ can be defined as $$\text{Nrd}(I)=\text{gcd}(\{\text{Nrd}(\alpha) | \alpha \in I \}).$$

Deuring \cite{MR5125} gave an equivalence of categories between supersingular $j$-invariants and maximal orders in the quaternion algebra $B_{p,\infty}$. Furthermore, if $E$ is a supersingular elliptic curve with $\text{End}(E)=\mathcal{O}$, then there is a one-to-one correspondence between isogenies $\phi:E\to E^{'}$ and left $\mathcal{O}$-ideals $I$. More details on the correspondence can be found in Chapter 42 of \cite{MR4279905}.

If the supersingular elliptic curve $E$ is defined over $\mathbb{F}_p$, then $\text{End}(E)$ contains an element $\alpha$ with minimal polynomial $x^2+p$ (see \cite{MR3451433}). Equivalently, the endomorphism ring $\text{End}(E)$ contains a subring $\mathbb{Z}[\sqrt{-p}]$. Ibukiyama \cite{MR683249} has given an explicit description of all maximal orders $\mathcal{O}$ in $B_{p, \infty}$ containing a root of $x^2+p=0$.

Choose a prime integer $q$ such that
\begin{align}\label{e1}
q \equiv 3 \pmod 8, \quad \left( \frac{p}{q} \right) =-1.
\end{align}
Then $B_{p,\infty}$ can be written as $B_{p,\infty}=\mathbb{Q}+\mathbb{Q}\alpha+\mathbb{Q}\beta+\mathbb{Q}\alpha\beta$ where $\alpha^2=-p$, $\beta^2=-q$ and $\alpha\beta=-\beta\alpha$. Let $r$ be an integer such that $r^2+p \equiv 0 \pmod  q$. Put
$$\mathcal{O}(q,r)=\mathbb{Z} + \mathbb{Z}\frac{1+\beta}{2} + \mathbb{Z} \frac{\alpha(1+\beta)}{2} + \mathbb{Z}\frac{(r+\alpha)\beta}{q}.$$
When $p \equiv 3 \pmod  4$, we choose an integer $r'$ such that $(r')^2+p \equiv 0 \pmod {4q}$ and put $$\mathcal{O}'(q,r')=\mathbb{Z} + \mathbb{Z}\frac{1+\alpha}{2} + \mathbb{Z} \beta + \mathbb{Z}\frac{(r'+\alpha)\beta}{2q}.$$

Note that the isomorphism class of $\mathcal{O}(q,r)$ or $\mathcal{O}'(q,r')$ depends on $q$, but does not depend on $r$ or $r'$. Ibukiyama's results \cite{MR683249} show that both $\mathcal{O}(q,r)$ and $\mathcal{O}'(q,r')$ are maximal orders of $B_{p,\infty}$ and the endomorphism ring $\text{End}(E)$ is isomorphic to $\mathcal{O}(q,r)$ or $\mathcal{O'}(q,r')$ with suitable choice of $q$ if $j(E) \in \mathbb{F}_p$.

Let $j \in \mathbb{F}_p$ be a supersingular $j$-invariant and $E(j)$ be the corresponding supersingular elliptic curve. We have $\text{End}(E(j)) \cong \mathcal{O}'(q,r')$ if $\frac{1+\pi}{2} \in \text{End}(E(j))$ and $\text{End}(E(j)) \cong \mathcal{O}(q,r)$ if $\frac{1+\pi}{2} \notin \text{End}(E(j))$ for some $q$ satisfying (\ref{e1}). Especially, we have $\text{End}(E(0)) \cong \mathcal{O}(3,1)$. By Lemma 1.8 and Proposition 2.1 of \cite{MR683249}, the following proposition holds.

\begin{proposition}\label{p1}(\cite[Lemma 2.9]{LI2020101619})
Suppose $q_{1} \neq q_{2}$ are primes satisfying $(\ref{e1})$. Let $K=\mathbb{Q}(\sqrt{-p})$. Suppose $q_{1}$ and $q_{2}$ have prime decompositions $q_{1} O_{K}=\mathfrak{q}_{1} \overline{\mathfrak{q}}_{1}$ and $q_{2} O_{K}=\mathfrak{q}_{2} \overline{\mathfrak{q}}_{2}$.

  (1) $\mathcal{O}(q_{1},r_1) \cong \mathcal{O}'(q_{2},r_2)$ if and only if $\left|\mathcal{O}\left(q_{1},r_1\right)^{\times}\right|=\left|\mathcal{O}^{\prime}\left(q_{2},r_2\right)^{\times}\right|=4$.

  (2) $\mathcal{O}^{\prime}\left(q_{1},r_1\right) \cong \mathcal{O}^{\prime}\left(q_{2},r_2\right)$ if and only if either $\mathfrak{q}_{1}\mathfrak{q}_{2} \in P(2)$ or $\mathfrak{q}_{1} \overline{\mathfrak{q}}_{2} \in P(2)$.

  (3) $\mathcal{O}\left(q_{1},r_1\right) \cong \mathcal{O}\left(q_{2},r_2\right)$ if and only if either $\mathfrak{q}_{1} \mathfrak{q}_{2} \in$ $P_{\mathbb{Z}}(2)$ or $\mathfrak{q}_{1} \overline{\mathfrak{q}}_{2} \in P_{\mathbb{Z}}(2)$.

\end{proposition}

\begin{remark}
For $p \equiv 3 \pmod 4$ and $p>3$, if $E$ is a supersingular elliptic curve defined over $\mathbb{F}_p$, then $|\operatorname{End}(E)^{\times}|=4$ if and only if $j(E)=1728$.
\end{remark}

For $p \equiv 3 \pmod 4$, $\mathfrak{q}_1 \mathfrak{q}_2 \in P(2)$ (resp. $\mathfrak{q}_{1} \overline{\mathfrak{q}}_{2} \in P(2)$) if and only if $\bar{\mathfrak{q}}_1$ and $\mathfrak{q}_2$ (resp. $\mathfrak{q}_1$ and $\mathfrak{q}_2$) are equivalent in the class group $C(-p)$ which means $q_1$ and $q_2$ can be represented by the same reduced forms with discriminant $-p$. Note that $q_i$ can be represented by  a form $f$ if and only if it can be represented by $f^{-1}$. We have the following corollary.

\begin{corollary} \label{c1}
Notations being as Proposition \ref{p1}.

(1) $\mathcal{O}'(q_{1},r_1) \cong \mathcal{O}'(q_{2},r_2)$ if and only if $q_1$ and $q_2$ can be represented by the same reduced forms with discriminant $-p$.

(2) $\mathcal{O}(q_{1},r_1) \cong \mathcal{O}(q_{2},r_2)$ if and only if $q_1$ and $q_2$ can be represented by the same reduced forms with discriminant $-16p$ (resp. $-4p$) for $p \equiv 1 \pmod 4$ (resp. $p \equiv 3 \pmod 4$).
\end{corollary}

\section{Endomorphism rings and binary quadratic forms}

By Ibukiyama's results \cite{MR683249}, there is a correspondence between supersingular $j$-invariants over $\mathbb{F}_p$ and maximal orders $\mathcal{O}(q,r)$'s and $\mathcal{O}'(q,r')$'s.

In this section, we will establish a correspondence between $\mathbb{F}_p$-isomorphism classes of supersingular elliptic curves and reduced quadratic forms with discriminant $-p$ or $-16p$.


We study the number of $\mathbb{F}_p$-roots of Hilbert class polynomials. Note that Hilbert class polynomials are class polynomials in \cite[\S 13]{MR3236783}. Let $D$ be the discriminant of an imaginary quadratic order $O$. The Hilbert class polynomial $H_D(X)$ is defined as
$$H_D(X):=\prod_{j(E) \in \text{Ell}_O(\mathbb{C})}(X-j(E)),$$
where $\text{Ell}_O(\mathbb{C})=\{j(E) \mid \text{End}(E/\mathbb{C}) \cong O\}$.
Write
$$J(D_1,D_2)=\prod_{\substack{[\tau_1],[\tau_2] \\ \text{disc}(\tau_i)=D_i}}(j(\tau_1)-j(\tau_2)),$$
where $[\tau_i]$ runs over all elements of the upper half-plane with discriminant $D_i$ modulo $\text{SL}_2(\mathbb{Z})$. We can view $J(D_1,D_2)$ as the resultant of the Hilbert class polynomials $H_{D_1}(X)$ and $H_{D_2}(X)$.

Gross and Zagier \cite{MR772491} studied the prime factorizations of $J(D_1,D_2)$ where $D_1$ and $D_2$ are two fundamental discriminants which are relatively prime. Recently, Lauter and Viray \cite{MR3431591} studied the prime factorizations of $J(D_1,D_2)$ for arbitrary $D_1$ and $D_2$.

\begin{lemma}\label{l1}
Let $p \equiv 3 \pmod 4$ be a prime. For a positive integer $1 \le a < p/4$, if $\left( \frac{a}{p} \right )=1$ and the equation $x^2 \equiv -p \pmod{4a}$ is solvable, then the Hilbert class polynomial $H_{-4a}(X) $ has $\mathbb{F}_p$-roots. Moreover, if $b$ is a solution of $x^2 \equiv -p \pmod{4a}$ with $-a<b\le a$, then the Hilbert class polynomials $H_{-4a}(X)$ and $H_{-4c}(X)$ ($c=\frac{b^2+p}{4a}$) have only one common root in $\mathbb{F}_p$.
\end{lemma}
\begin{proof}
If $a=1$  and $p \equiv 3 \pmod 4$, then the Hilbert class polynomial $H_{-4}(X)$ has only one root $1728$ and $b=1$ is a solution of the equation $x^2 \equiv -p \pmod{4}$. In this case, $c=\frac{p+1}{4}$ and we can show that $H_{-4}(X)$ and $H_{-(p+1)}(X)$ have only one common root in $\mathbb{F}_p$ by \cite[Theorem 1.5]{MR3431591}. In the following, we assume $1<a<p/4$.

If $2 \mid a$, then $p \equiv 7 \pmod 8$ and $\left( \frac{2}{p} \right )=1$. For any odd prime $\ell \mid a$, $x^2 \equiv -p \pmod{4a}$ is solvable if and only if $\left( \frac{-p}{\ell} \right )=\left( \frac{\ell}{p} \right )=1$.

Since $4a<p$, the prime $p$ does not divide the discriminant of $H_{-4a}(X)$ and \cite[Theorem 1.1]{IJNT} holds. Let $F$ be the genus field of $\mathbb{Q}(\sqrt{-4a})$ and $E=F\cap \mathbb{R}$ the maximal real subfield of $F$. We have that $x^2 \equiv -p \pmod{4a}, 4<4a < p  $ is solvable if and only if $p$ splits completely in $E$. It follows that the Hilbert class polynomial $H_{-4a}(X)$ has $\mathbb{F}_p$-roots.

If $-a<b\le a$ is a solution of $x^2 \equiv -p \pmod{4a}$ and $c=\frac{b^2+p}{4a}$, then the equation $x^2 \equiv -p \pmod{4c}$ is solvable. Moreover, we have $p>4c$ since $4<4a<p$, so the Hilbert class polynomial $H_{-4c}(X)$ has $\mathbb{F}_p$-roots.

We have $p \mid (4a\cdot 4c-x^2)$ if and only if $x=\pm 2b$. By \cite[Theorem 1.5]{MR3431591}, we have $p \parallel J(-4a,-4c)$. It follows that $H_{-4a}(X)$ and $H_{-4c}(X)$ have only one common root in $\mathbb{F}_p$.

\end{proof}

\begin{remark}
Let $(a,b,c)$ be a reduced form with discriminant $-p$. We have that $H_{-4a}(X)$ and $H_{-4c}(X)$ have only one common $\mathbb{F}_p$-root $j$. Moreover, $4a=|D_1|$ is the smallest positive integer such that $j$ is a $\mathbb{F}_p$-root of $H_{D_1}(X)$, and $4c=|D_2|$ is the smallest positive integer such that $H_{-4a}(X)$ and $H_{D_2}(X)$ have a common $\mathbb{F}_p$-root $j$.
\end{remark}

Let $p \equiv 3 \pmod 4$ be a prime. Let $E$ be a supersingular elliptic curve over $\mathbb{F}_p$ with $\operatorname{End}(E) \cong \mathcal{O}'(q,r')$. If $j(E)$ is a $\mathbb{F}_p$-root of the Hilbert class polynomial $H_D(X)$ with $|D|<p$, then there exist two forms $(q,\pm r',\frac{(r')^2+p}{4q})$ which can represent $|D|/4$ by the proof of \cite[Theorem 1.3]{IJNT}.

\begin{theorem}\label{zt1}
Let $p \equiv 3 \pmod 4$ be a prime with $p\ge 7$. Given a reduced form with discriminant $-p$ which can represent a prime $q$ satisfying (\ref{e1}), there exists a unique supersingular $j$-invariant $j \in \mathbb{F}_p$ such that $\operatorname{End}(E(j)) \cong \mathcal{O}'(q,r')$. Moreover, the inverse form corresponds to the same supersingular $j$-invariant.
\end{theorem}
\begin{proof}
Let $(a,b,c)$ be a reduced quadratic form with discriminant $-p$ which can represent a prime $q$ satisfying (\ref{e1}). Since $b^2-4ac=-p$ then $1 \le a < \sqrt{p/3} < p/4$, and we have that $H_{-4a}(X)$ mod $p$ and $H_{-4c}(X)$ mod $p$ have a unique common root $j \in \mathbb{F}_p$ by Lemma \ref{l1}. By Ibukiyama's theorems, the endomorphism ring of $E(j)$ is isomorphic to the maximal order $\mathcal{O}(q_j,r_j)$ or $\mathcal{O}'(q_j,r'_j)$ for some $q_j$ satisfying (\ref{e1}). For simplicity, we can assume $q_j=q$.

If $a=1$, then $j=1728$ and we can assume $\operatorname{End}(E(1728)) \cong \mathcal{O}'(q,r') \cong \mathcal{O}(q,r)$. In the following, we assume $1<a<c<p/4$.

Suppose $\operatorname{End}(E(j)) \cong \mathcal{O}(q,r)$. Since $j$ is the common root of $H_{-4a}(X)$ mod $p$ and $H_{-4c}(X)$ mod $p$, we have that the form $(q,4r,\frac{4r^2+4p}{q})$ can represent $4a$ and $4c$ by the proof of Theorem 1 in \cite{MR1040429}. Moreover, the number $4a=|D_1|$ is the smallest positive integer such that $j$ is a $\mathbb{F}_p$-root of $H_{D_1}(X)$, and the number $4c=|D_2|$ is the smallest positive integer such that $H_{D_1}(X)$ mod $p$ and $H_{D_2}(X)$ mod $p$ have a common $\mathbb{F}_p$-root $j$. The first (resp. second) successive minimum of $(q,4r,\frac{4r^2+4p}{q})$ is $\sqrt{4a}$ (resp. $\sqrt{4c}$). By Proposition 5.7.3 and Theorem 5.7.6 in \cite{MR2300780}, the form $(q,4r,\frac{4r^2+4p}{q})$ is equivalent to the form $(4a,4b,4c)$ or $(4a,-4b,4c)$, which is a contradiction. We must have $\operatorname{End}(E(j)) \cong \mathcal{O}'(q,r')$.

So, we have $\text{End}(E(j)) \cong \mathcal{O}'(q,r')$. Obviously, the inverse form $(a,-b,c)$ corresponds the same supersingular $j$-invariant.

\end{proof}


For the maximal orders $\mathcal{O}(q,r)$'s, a similar result also holds, while the corresponding forms are of discriminant $-16p$.

We review some facts about genera of positive definite quadratic forms. Let $p>3$ be a prime. The primitive forms with discriminant $-16p$ have two assigned characters $\chi(a)= \left( \frac{a}{p} \right)$ and $\delta(a)=(-1)^{(a-1)/2}$. Furthermore, the primitive forms with discriminant $-16p$ can be divided into two genera (see \cite[Theorem 3.15]{MR3236783}). The genus which contains the principal form $(1,0,4p)$ is called the principal genus, and we call the other non-principal genus.

Let $(a,b,c)$ be a reduced form with discriminant $-16p$ in the non-principal genus. If $a$ (resp. $c$) is odd, then $a\equiv 3 \pmod 4$ (resp. $c\equiv 3 \pmod 4$). If $a$ is even, we claim $4\mid a$. If $a \equiv 2 \pmod 4$, then $8 \mid b^2(=4ac-16p)$. It follows $4 \mid b$ and $2 \mid c$. This is a contradiction since the form $(a,b,c)$ is primitive. Similarly, if $c$ is even, then $4 \mid c$.

Similar to Lemma \ref{l1}, we have the following lemma.

\begin{lemma}\label{l2}
Let $p$ be a prime and $a \equiv 0$ or $3 \pmod 4$. If $\left( \frac{-a}{p} \right)=-1$ and the equation $x^2 \equiv -16p \pmod{4a}$ is solvable, then the Hilbert class polynomial $H_{-a}(X)$ has $\mathbb{F}_p$-roots. Moreover, let $b$ be a solution of $x^2 \equiv -16p \pmod{4a}$ with $-a < b \le a$. If the form $(a,b,c)$ ($c=\frac{b^2+16p}{4a}$) is primitive, then the Hilbert class polynomials $H_{-a}(X)$ and $H_{-c}(X)$ have only one common root in $\mathbb{F}_p$.
\end{lemma}
Let $E$ be a supersingular elliptic curve over $\mathbb{F}_p$ with $\operatorname{End}(E) \cong \mathcal{O}(q,r)$. If $j(E)$ is a $\mathbb{F}_p$-root of the Hilbert class polynomial $H_D(X)$ with $|D|<p$, then there exist two quadratic forms $(q,\pm 4r,\frac{4r^2+4p}{q})$ which can represent $|D|$.

\begin{theorem}\label{zt2}
Let $p \ge 5$ be a prime. Given a reduced form with discriminant $-16p$ in the non-principal genus which can represent a prime $q$ satisfying $(\ref{e1})$, there is a unique supersingular $j$-invariant in $\mathbb{F}_p$ with $\operatorname{End}(E(j)) \cong \mathcal{O}(q,r)$. Moreover, the inverse form corresponds to the same supersingular $j$-invariant.
\end{theorem}

\begin{proof}

If $(a,b,c)$ is a primitive reduced form in the non-principal genus with discriminant $b^2-4ac=-16p$, then $\left( \frac{-a}{p} \right)=-1$ and $a, c \equiv 0,3 \pmod 4$. By Lemma \ref{l2}, the Hilbert class polynomials $H_{-a}(X)$ and $H_{-c}(X)$ mod $p$ have one common root $j \in \mathbb{F}_p$.

If $p \equiv 3 \pmod 4$ and $a=4$, then $j=1728$. If $ p \equiv 2 \pmod 3$ and $a=3$, then $j=0$ and $\operatorname{End}(E(j)) \cong \mathcal{O}(3,1)$. If $p=5$, then there is only one supersingular $j$-invariant $0$ in $\mathbb{F}_p$. In the following, we assume $p>5$ and $a>4$, so $a<c<p$.

As in the proof of Theorem \ref{zt1}, we can show that $\text{End}(E(j)) \cong \mathcal{O}(q,r)$. Moreover, we can also prove $\mathbb{Z}[(1+\sqrt{-p})/2]$ cannot be embedded into $\mathcal{O}(q,r)$.

We claim that $j$ is not a root of the Hilbert class polynomial $H_{-p}(X)$ mod $p$. Since $4 \nmid \gcd(a,c)$, we can assume $4 \nmid a$ and $a<p$. If there exist integers $x$ and $k \ge 0$ such that $4pk= ap-x^2$, then $x^2=p(a-4k)$. It follows that $p \mid x$ and $p \mid (a-4k)$. This is a contradiction since $0<a-4k \le a <p$. Let $d_1=-a$ and $d_2=-p$. We have $v_p(F(\frac{d_1d_2-x^2}{4}))=0$ for any $|x|<\sqrt{d_1d_2}$ by Theorem 1.1 in \cite{MR3431591}. This shows that the Hilbert class polynomials $H_{-a}(X)$ and $H_{-p}(X)$ have no common roots in $\mathbb{F}_p$, and $j$ is not a root of $H_{-p}(X)$ mod $p$. It follows that the order $\mathbb{Z}[(1+\sqrt{-p})/{2}]$ cannot be embedded into the endomorphism ring of $E(j)$, so we can assume $\operatorname{End}(E(j)) \cong \mathcal{O}(q_j,r_j)$ for some $q_j$ satisfying (\ref{e1}).

We have $\text{End}(E(j)) \cong \mathcal{O}(q,r)$. Obviously, the inverse form $(a,-b,c)$ corresponds to the same $j$-invariant.

\end{proof}

\begin{remark}
Given a maximal order $\mathcal{O}$ in $B_{p,\infty}$, one can construct a supersingular elliptic curve $E$ such that $\operatorname{End}(E) \cong \mathcal{O}$ by taking gcds of the reductions
modulo $p$ of three Hilbert class polynomials in \cite{MR3240797}. For the maximal order $\mathcal{O}(q,r)$ (resp. $\mathcal{O}'(q,r')$), one can construct a supersingular elliptic curve $E$ over $\mathbb{F}_p$ by taking gcds of the reductions modulo $p$ of two Hilbert class polynomials.
\end{remark}

If $p \equiv 3 \pmod 4$, then the reduced form $(1,1,(p+1)/4)$ is the unique one with order $\le 2$ in the class group $C(-p)$. Moreover, the reduced form $(4,0,p)$ is the unique one in the non-principal genus with order $\le 2$ in the class group $C(-16p)$. In fact, the forms $(1,1,(p+1)/4)$ and $(4,0,p)$ correspond to the same supersingular $j$-invariant $1728$. If $p\equiv 1 \pmod 4$, then there exists no form in the non-principal genus with order $\le 2$ in the class group $C(-16p)$.

For a prime $p > 5$, let $S_p$ be the set of all supersingular $j$-invariants in $\mathbb{F}_p$. If $p \equiv 1 \pmod 4$, the number of reduced forms with discriminant $-16p$ in the non-principal genus is $h(-16p)/2=h(-4p)$ and $\# S_p=h(-4p)/2$. If $p \equiv 3 \pmod 4$, $\#S_p=(h(-4p)+1)/2+(h(-p)+1)/2-1=h(-4p)/2+h(-p)/2$. So we have the following formula:

$$ \# S_{p}= \left\{
\begin{array}{lcl}
h(-4p)/2 & &{\text{if} \ p\equiv 1 \pmod{4} ,} \\
h(-p) & &{\text{if} \ p\equiv 7 \pmod{8} ,}\\
2h(-p) & &{\text{if} \ p\equiv 3 \pmod{8} ,}
\end{array}
\right.
$$
where $h(D)$ is the class number of the imaginary quadratic order with discriminant $D$. This formula coincides with the result in \cite{MR3451433}.

In fact, Theorems \ref{zt1} and \ref{zt2} establish a correspondence between reduced quadratic forms and supersingular $j$-invariants in $\mathbb{F}_p$. Notice that this correspondence is not $1$-to-$1$. In the following of this section, we consider twist curves and isomorphism classes of supersingular elliptic curves over $\mathbb{F}_p$.

Let $p>3$ be a prime. Choose a non-square element $\mu \in \mathbb{F}_p$. For an elliptic curve $E:Y^2=X^3+aX+b$, we define its twist curve $\tilde{E}$ by the following way:
\begin{equation*}
\begin{array}{llll}
 (1) & \tilde{E}:Y^2=X^3+a\mu^2X+b\mu^3 & \text{if} \ & j(E) \neq 1728; \\
 (2) & \tilde{E}:Y^2=X^3+a\mu X & \text{if} \ & j(E) = 1728.
\end{array}
\end{equation*}

This definition is a simple case in Proposition 5.4 of \cite[X.5]{MR2514094}. The twist curves have the same $j$-invariant, while they are not isomorphic over $\mathbb{F}_p$.

Let $\text{Iso}(\mathbb{F}_p,\frac{1+\pi}{2})$ (resp. $\text{Iso}(\mathbb{F}_p,\pi)$) denote the set of $\mathbb{F}_p$-isomorphism classes of supersingular elliptic curves over $\mathbb{F}_p$ with endomorphism ring $\mathcal{O}'(q,r')$ (resp. $\mathcal{O}(q,r)$). We have the following theorem.

\begin{theorem}
Let $p>3$ be a prime. Every class in $\operatorname{Iso}(\mathbb{F}_p,\frac{1+\pi}{2})$ is $1$-to-$1$ corresponding to a reduced quadratic form with discriminant $-p$. Every class in $\operatorname{Iso}(\mathbb{F}_p,\pi)$ is $1$-to-$1$ corresponding to a reduced form in the non-principal genus with discriminant $-16p$.
\end{theorem}

\begin{proof}
For $j(E) \neq 1728$, the $\mathbb{F}_p$-isomorphism class containing $E$ corresponds to the form $(a,b,c)$, and the $\mathbb{F}_p$-isomorphism class containing $\tilde{E}$ corresponds to $(a,-b,c)$.

As we know, the root of the Hilbert class polynomial $H_{-4}(X)=X-1728$ is a supersingular $j$-invariant if and only if $p \equiv 3 \pmod 4$. In this case, we use $E(1728):Y^2=X^3-X$ to represent an isomorphism class in $\mathbb{F}_p$. Its twist curve can be written as $\tilde{E}(1728):Y^2=X^3+4X$. Furthermore, the curve $E(1728)$ is in the surface of the $\mathbb{F}_p$ isogeny graph, and the curve $E(1728)$ corresponds to the form $(1,1,(p+1)/4)$. The curve $\tilde{E}(1728)$ is on the floor, and it corresponds to the form $(4,0,p)$.

\end{proof}

\begin{example}\label{ex1}
If $p=83$, then there are 12 $\mathbb{F}_{83}$-isomorphism classes of supersingular elliptic curves. The curves and the corresponding forms are as following.

\small
$$\begin{array}{llll}
  E_0:y^2=x^3-x & f_0=(1,1,21) &\quad \tilde{E}_0:y^2=x^3+x &g_0=(4,0,83) \\
  E_1:y^{2}=x^{3}+13 x^{2}-x & f_1=(3,1,7) &\quad \tilde{E}_1:y^{2}=x^{3}-13 x^{2}-x &f_1^{-1}=(3,-1,7) \\
  E_2:y^{2}=x^{3}+11 x^{2}+x & g_2=(11,-6,31) & \quad \tilde{E}_2:y^{2}=x^{3}-11 x^{2}+x & g_2^{-1}=(11,6,31) \\
  E_3:y^{2}=x^{3}+12 x^{2}+x & g_3=(7,4,48) & \quad \tilde{E}_3:y^{2}=x^{3}-12 x^{2}+x & g_3^{-1}=(7,-4,48)\\
  E_4:y^{2}=x^{3}+ \ 6 x^{2}+x & g_4=(16,-12,23) & \quad \tilde{E}_4:y^{2}=x^{3}-\ 6 x^{2}+x & g_4^{-1}=(16,12,23) \\
  E_5:y^{2}=x^{3}-13 x^{2}+x & g_5=(3,-2,111) & \quad \tilde{E}_5:y^{2}=x^{3}+13 x^{2}+x & g_5^{-1}=(3,2,111)
\end{array}$$

\normalsize
The supersingular elliptic curves are from Remark 5 in \cite{MR3897883}. Note that the forms $f_i$ and $f_i^{-1}$ are of discriminant $-83$, and the forms $g_i$ and $g_i^{-1}$ are of discriminant $-83\times 16$. In fact, the corresponding form is not unique except for $E_0$ and $\tilde{E}_0$. We just choose one correspondence in this example.
\end{example}

We have established a correspondence between the set of $\mathbb{F}_p$-isomorphism classes of supersingular elliptic curves over $\mathbb{F}_p$ and the set of reduced quadratic forms. In fact, this correspondence is not canonical. Moreover, the reduced forms with discriminant $-p$ form a group by composition. This suggests us to research isogenies between supersingular elliptic curves over $\mathbb{F}_p$.

\section{Isogenies and binary quadratic forms}

As we have seen, the correspondence which we have established in Section 3 is only about sets. Now, let us consider $\mathbb{F}_p$-isogenies between supersingular elliptic curves defined over $\mathbb{F}_p$. Note that these isogenies correspond to ideals in imaginary quadratic orders, hence they could be represented by quadratic forms.

In this section, we present the connection between quadratic forms corresponding to isogenies over $\mathbb{F}_p$ and quadratic forms corresponding to supersingular elliptic curves defined over $\mathbb{F}_p$.

To study the connections between $\mathcal{O}(q,r)$ and $\mathcal{O}'(q,r')$, we recall the connections between quadratic forms with different discriminants.

\subsection{Quadratic forms with different discriminants}
We now review some basic facts about quadratic forms with different discriminants. If
$$
S=\left(\begin{array}{ll}
\alpha & \beta \\
\gamma & \delta
\end{array}\right)
$$
is a $2 \times 2$ matrix with integer coefficients and determinant $\alpha\delta-\beta\gamma=s>0$, then the change of variables takes a form $f=(a, b, c)$ of discriminant $D$ to a form
$$
f'=\left(a', b', c' \right)=\left(a \alpha^{2}+b \alpha \gamma+c \gamma^{2}, b(\alpha \delta+\beta \gamma)+2(a \alpha \beta+c \gamma \delta), a \beta^{2}+b \beta \delta+c \delta^{2}\right)
$$
of discriminant $D s^{2}$. In matrix notation this is
$$
\left(\begin{array}{cc}
a^{\prime} & b^{\prime} / 2 \\
b^{\prime} / 2 & c'
\end{array}\right)=\left(\begin{array}{ll}
\alpha & \gamma \\
\beta & \delta
\end{array}\right)\left(\begin{array}{cc}
a & b / 2 \\
b / 2 & c
\end{array}\right)\left(\begin{array}{ll}
\alpha & \beta \\
\gamma & \delta
\end{array}\right),
$$
which we write as $f'=S^\mathrm{T} f S$ for brevity. We shall call such a matrix $S$ a transformation of determinant $s$ and shall say that $f'$ is derived from $f$ by the transformation $S$ of determinant $s$.

Any primitive form of discriminant $Ds^2$ can be derived from a primitive form of discriminant $D$ by the application of a transformation $S$ of determinant $s$. We define two transformations $S_1$ and $S_2$ of discriminant $s$ to be $\operatorname{SL}_2(\mathbb{Z})$-right-equivalent if there exists a matrix $M \in \operatorname{SL}_2(\mathbb{Z})$ such that $S_1M=S_2$. Moreover, we have the following proposition.
\begin{proposition}(\cite[Proposition 7.2]{MR1012948})
Let $s$ be a prime integer. The $\operatorname{SL}_2(\mathbb{Z})$-right-equivalent transformations have as equivalence class representatives the $s+1$ transformations
\begin{center}
$S_s=\left(\begin{array}{ll}
1 & 0 \\
0 & s
\end{array}\right)$
\ and  \
$S_i=\left(\begin{array}{ll}
s & i \\
0 & 1
\end{array}\right),$
\end{center}
for $0 \le i \le s-1$.
\end{proposition}

\begin{remark}
Let $q$ be a prime with $q^2 \nmid Ds^2$. If $q$ can be represented by $f$, then it must be represented by one of $S^\mathrm{T}_i f S_i$ with $0 \le i \le s$. On the contrary, if a prime $q$ can be represented by one of $S^\mathrm{T}_i f S_i$, then $q$ must be represented by $f$.
\end{remark}

\begin{proposition}
Let $p\equiv 3 \pmod 4$ be a prime. If $q_1 \neq q_2$ are primes satisfying $(\ref{e1})$, then $\mathcal{O}(q_1,r_1) \cong \mathcal{O}(q_2,r_2)$ if and only if $q_1$ and $q_2$ can be represented by the same primitive reduced quadratic form with discriminant $-16p$.
\end{proposition}
\begin{proof}
For $p\equiv 3 \pmod 4$, $\mathcal{O}(q_1,r_1) \cong \mathcal{O}(q_2,r_2)$ if and only if $q_1$ and $q_2$ can be represented by the same primitive reduced quadratic form with discriminant $-4p$ by Corollary \ref{c1}.

Assume that $q_1$ and $q_2$ can be represented by a primitive form $(a,b,c)$ with discriminant $-4p$. For $s=2$, the following three forms with discriminant $-16p$ are derived from $(a,b,c)$:
$$(a,2b,4c) \quad (4a,2b,c) \quad (4a,4a+2b,a+b+c).$$
For $p \equiv 3 \pmod 4$ and $2 \mid b$, we have $ac=\frac{b^2+4p}{4}=\left( \frac{b}{2} \right)^2+p$. It is easy to show that only one of the three forms above can represent $q_1$ or $q_2$.

We show that $q_1$ and $q_2$ can be represented by the same reduced quadratic form with discriminant $-16p$. On the contrary, if $q_1$ and $q_2$ can be represented by the same reduced quadratic form with discriminant $-16p$, then $\mathcal{O}(q_1,r_1) \cong \mathcal{O}(q_2,r_2)$ obviously.

\end{proof}

In the following of this subsection, we discuss reduced primitive quadratic forms with discriminant $-16p$ derived from forms with discriminant $-p$.

\begin{proposition}\label{p2}
If $p \equiv 3 \pmod 8$, then every reduced form with discriminant $-p$ divides into six reduced forms with discriminant $-16p$ and three of them are in the non-principal genus. If $p\equiv 7 \pmod 8$, then every reduced form with discriminant $-p$ divides into two reduced forms with discriminant $-16p$ and one of them is in the non-principal genus.
\end{proposition}

\begin{proof}
Let $(a,b,c)$ be a reduced form with discriminant $-p$. The following three forms with discriminant $-4p$ are derived from $(a,b,c)$.

$$(a,2b,4c) \quad (4a,2b,c) \quad (4a,4a+2b,a+b+c)$$

(1) If $p \equiv 3 \pmod 8$, then $ac$ is odd. It follows that the above three forms are primitive. We have $h(-4p)=3h(-p)$, so these three forms are not equivalent in $C(-4p)$. The following nine forms with discriminant $-16p$ are derived from $(a,b,c)$.
$$\begin{array}{lll}
  (4a,4b,4c)         & (4a,4b,4c)           & (4a,8a+4b,4a+4b+4c) \\
  (a,4b,16c)         & (16a,4b,c)           & (16a,8a+4b,a+b+c) \\
  (4a,4a+4b,a+2b+4c) & (16a,16a+4b,4a+2b+c) & (16a,24a+4b,9a+3b+c)
\end{array}$$

It is obvious that the three forms in the first row are not primitive, and other six forms are not equivalent in $C(-16p)$ since $h(-16p)=2h(-4p)$. Moreover, only one form in every column is in the non-principal genus.

(2) If $p \equiv 7 \pmod 8$, then $ac$ is even. Only one of the three forms is primitive. Without loss of generality, we can assume the form $(a,2b,4c)$ is primitive. In this case, one of $(a,4b,16c)$ and $(4a,4a+4b,a+2b+4c)$ is in the non-principal genus.

\end{proof}

\subsection{Isogenies and quadratic forms}
Denote the imaginary quadratic field $K=\mathbb{Q}(\sqrt{-p})$. Let $O$ be an order in $K$ and $O_K$ be the integer ring of $K$.

We always assume that $E_1$ and $E_2$ are supersingular elliptic curves defined over $\mathbb{F}_p$ and  the isogeny $\phi : E_1 \to E_2$ is also defined over $\mathbb{F}_p$ in this following. Without loss of generality, we can restrict to the case when $\text{deg}(\phi)=\ell$ is a prime, since every isogeny can be written as composition of isogenies of prime degree. We divide the remainder of this subsection into three cases according to the endomorphism rings of $E_1$ and $E_2$.
\\

\textbf{Case 1:} $\operatorname{End}_{\mathbb{F}_p}( E_1) \cong \operatorname{End}_{\mathbb{F}_p}( E_2) \cong \mathbb{Z}[(1+\sqrt{-p})/{2}]$.
\\

In this case, we must have $p \equiv 3 \pmod 4$. We write $\mathcal{O}'(q,r)$ with $r^2 +p \equiv 0 \pmod {4q}$. If $\operatorname{End} (E_1) \cong  \mathcal{O}'(q_1,r_1)$, without loss of generality, we can assume that $E_1$ corresponds to the form $(q_1,r_1, \frac{r_1^{2}+p}{4q_1})$. Since $\ell$ splits in $O_K$, there exists an integer $-\ell<b\le \ell$ satisfying $b^2 \equiv -p \pmod {4\ell}$. We also assume $\ell O_K=\mathfrak{l}\bar{\mathfrak{l}}$, where $\mathfrak{l}=[\ell, (-b+\sqrt{-p})/2]$ (see \cite[Theorem 6.15]{MR1012948}).

\begin{theorem}\label{zt4}
For a prime $p \equiv 3 \pmod 4$, let $E_1$ be a supersingular elliptic curve with $\operatorname{End} (E_1) \cong  \mathcal{O}'(q_1,r_1)$. If $\phi : E_1 \to E_2$ is an isogeny defined over $\mathbb{F}_p$ of degree $\ell$ with $\operatorname{End}_{\mathbb{F}_p}(E_2) \cong O_K$, then there is a prime $q_2$ satisfying $(\ref{e1})$ which can be represented by the form $(q_1,r_1, \frac{r_1^{2}+p}{4q_1})(\ell, b, \frac{b^2+p}{4\ell})^2 $ such that $\operatorname{End}(E_2) \cong \mathcal{O}'(q_2,r_2)$.
\end{theorem}

\begin{proof}
For $p \equiv 3 \pmod 4$, if $\operatorname{End} (E_1) \cong  \mathcal{O}'(q_1,r_1)$, then we can assume $E_1$ corresponds to the quadratic form $(q_1,r_1, \frac{r_1^{2}+p}{4q_1})$. By \cite[Theorem 6.15]{MR1012948}, we can assume $q_1O_K=\mathfrak{q}_1 \bar{\mathfrak{q}}_1$ with $\bar{\mathfrak{q}}_1 =[q_1,\frac{-r_1+\sqrt{-p}}{2}]$.

If $\phi : E_1 \to E_2$ is an isogeny defined over $\mathbb{F}_p$ and $\text{End}_{\mathbb{F}_p}(E_1) \cong \text{End}_{\mathbb{F}_p}(E_2) \cong O_K$, then we can assume $\text{Ker}(\phi)=\{ P \in E_1[\ell] \mid [(-b+\sqrt{-p})/{2}](P)=\infty \}$ with $b^2 \equiv -p \pmod {4\ell}$. Note that $\mathfrak{l}=[\ell, (-b+\sqrt{-p})/2]$ is an ideal in $O_K$, and it corresponds to the form $(\ell, b, \frac{b^2+p}{4\ell})$ by the isomorphism in Proposition \ref{t3}. It follows that $\mathcal{O}'(q_1,r_1)\mathfrak{l}$ is the ideal kernel of $\phi$. Moreover, $\mathfrak{l}^{-1} \mathcal{O}'(q_1,r_1) \mathfrak{l}$ is the right order of $\mathcal{O}'(q_1,r_1)\mathfrak{l}$ with $\mathfrak{l}^{-1}=\bar{\mathfrak{l}}/ \ell$.

By \cite[Proposition 4.2]{MR683249}, we have $\text{End}(E_2) \cong \mathcal{O}'(q_2,r_2)$ for any $q_2$ satisfying (\ref{e1}) which is the norm of a prime ideal $\mathfrak{q}_2 \in \bar{\mathfrak{q}}_1 \mathfrak{l}^{2} P(O_K)$. The ideal $\bar{\mathfrak{q}}_1=[q_1,\frac{-r_1+\sqrt{-p}}{2}]$ corresponds to the form $(q_1,r_1, \frac{r_1^{2}+p}{4q_1})$, so $E_2$ corresponds to the form $(q_2,r_2, \frac{r_2^{2}+p}{4q_2}) \sim (q_1,r_1, \frac{r_1^{2}+p}{4q_1})(\ell, b, \frac{b^2+p}{4\ell})^2$ and $\operatorname{End}(E_2) \cong \mathcal{O}'(q_2,r_2)$.

\end{proof}

\begin{remark}
For $\left( \frac{-p}{\ell} \right)=1$, there are two $\ell$-isogenies $\phi:E_1 \to E_2$ and $\phi':E_1 \to E'_2$ over $\mathbb{F}_p$ with $\operatorname{End}_{\mathbb{F}_p}(E_2) \cong \operatorname{End}_{\mathbb{F}_p}(E'_2) \cong O_K$. The isogeny $\phi'$ corresponds to the form $(\ell, -b, \frac{b^2+p}{4\ell})$, and $E'_2$ corresponds to the form $(q_1,r_1, \frac{r_1^{2}+p}{4q_1})(\ell, -b, \frac{b^2+p}{4\ell})^2$.
\end{remark}

\textbf{Case 2:} $\operatorname{End}_{\mathbb{F}_p}( E_1) \cong \operatorname{End}_{\mathbb{F}_p}( E_2) \cong \mathbb{Z}[\sqrt{-p}]$.
\\

If $\text{End}(E_1) \cong \mathcal{O}(q_1,r_1)$ and $\text{End}(E_2) \cong \mathcal{O}(q_2,r_2)$ and $\text{deg}(\phi)=\ell$ is an odd prime, we have the following proposition.

\begin{proposition}\label{zt3}
Let $\mathfrak{l}$ be an integral ideal of $O=\mathbb{Z}[2\sqrt{-p}]$ where $N(\mathfrak{l})=\ell$ is an odd prime. Put $\mathfrak{q}_1=[q_1,2r_1+2\sqrt{-p}]$. We have an isomorphism: $\mathfrak{l}^{-1} \mathcal{O}(q_1,r_1) \mathfrak{l} \cong \mathcal{O}(q_2,r_2)$ for some $q_2$ satisfying $(\ref{e1})$ which is the norm of a prime ideal $\mathfrak{q}_2 \in \bar{\mathfrak{q}}_1 \mathfrak{l}^{2} P_{\mathbb{Z}}(2)$ (resp. $\mathfrak{q}_2 \in \bar{\mathfrak{q}}_1 \mathfrak{l}^{2} P_{\mathbb{Z}}(4)$) if $p \equiv 1 \pmod 4$ (resp. $p \equiv 3 \pmod 4$).
\end{proposition}

\begin{proof}
We can assume $\mathfrak{l}=[\ell,-b+2\sqrt{-p}]$ with $b^2 \equiv -4p \pmod {\ell}$.
We have $\mathfrak{l}^{-1} \mathcal{O}(q_1,r_1) \mathfrak{l} \subset(O / 2)+\mathfrak{q}_1 \bar{\mathfrak{l}}^{2} \beta_1 / 4 q_1 \ell$. We can take an isomorphism $\mathfrak{l}^{-1} O(q_1,r_1) \mathfrak{l} \cong \mathcal{O}(q_2,r_2)$ for some $q_2$ satisfying $(\ref{e1})$ such that the image of $O$ is $O$. Let $\beta_2$ be the image of $\beta_{q_2}$ of this isomorphism. So we have $\beta_2 \in \mathfrak{q}_1 \bar{\mathfrak{l}}^{2} \beta_1 / 4 q_1 \ell$. Put $\kappa=\ell q_1 \beta_2 \beta_1^{-1}$. Then $N(2\kappa)=4\ell^{2} q_1 q_2 \in \mathbb{Z}$ and $\operatorname{tr}(2\kappa) \in \mathbb{Z}$, since $2\kappa \in \mathfrak{q}_1 \bar{\mathfrak{l}}^{2}/2 \subset O / 2 $. It follows that $2 \kappa \in O_K$.

If $p \equiv 1 \pmod 4$, we can write $2\kappa =a+b\sqrt{-p}$ with $a, b \in \mathbb{Z}$, and we have $N(2\kappa)=a^2+pb^2$. Since $4 \mid N(2k)$ and $p \equiv 1 \pmod 4$, we must have $2 \mid a$ and $2 \mid b$. This means $\kappa \in O_K$.

If $p \equiv 3 \pmod 4$, then $2 \kappa \in O/2 \cap O_K =\mathbb{Z}+\mathbb{Z} \sqrt{-p}$. Write $2\kappa =a+b\frac{1+\sqrt{-p}}{2}$ with $2 \mid b$, so $N(2\kappa)=a^2+ab+\frac{1+p}{4} b^2 \equiv a^2+ab \pmod 4$. We have $2 \mid a$ since $4 \mid N(2\kappa)$. It follows that $\kappa \in O_K$.

Moreover, we have
$$
\left(1+\beta_2\right) / 2 \in \mathfrak{l}^{-1} \mathcal{O}(q_1,r_1) \mathfrak{l}, \quad \text { so } \quad  (\ell q_1+\kappa \beta_1) / 2 q_1 \in \ell \mathfrak{l}^{-1} \mathcal{O}(q_1,r_1) \mathfrak{l} \subset \mathcal{O}(q_1,r_1).
$$
If $p \equiv 1 \pmod 4$, putting $\kappa=c+d \sqrt{-p}$, we have
$$
(\ell q_1+\kappa \beta_1) / 2 q_1=(\ell q_1-c+d r) / 2 q_1+(c-d r)(1+\beta_1) / 2 q_1+d(r+\sqrt{-p}) \beta_1 / 2 q_1.
$$
So we have $2 \mid d$ and $(\kappa) \in P_{\mathbb{Z}}(2)$. Furthermore, since $N(\kappa) =\ell^2 q_1 q_2$ and $\mathfrak{q}_1 \bar{\mathfrak{l}}^2 \mid (\kappa)$, there is an integral ideal $\mathfrak{q}_2$ such that $(\kappa)=\mathfrak{q}_1 \mathfrak{q}_2 \bar{\mathfrak{l}}^{2}$ and $N\left(\mathfrak{q}_2\right)=q_2$.

If $p \equiv 3 \pmod 4$, putting $\kappa=c+d \frac{1+\sqrt{-p}}{2}$, we have
$$
(\ell q_1+\kappa \beta_1) / 2 q_1=(2\ell q_1-2c-d+d r) / 4 q_1+(2c+d-d r)(1+\beta_1) / 4 q_1+d(r+\sqrt{-p}) \beta_1 / 4q_1.
$$
So we have $4 \mid d$ and $(\kappa) \in P_{\mathbb{Z}}(4)$. Similarly, there is an integral ideal $\mathfrak{q}_2$ such that $(\kappa)=\mathfrak{q}_1 \mathfrak{q}_2 \bar{\mathfrak{l}}^{2}$ and $N\left(\mathfrak{q}_2\right)=q_2$.

\end{proof}

If $p \equiv 1 \pmod 4$ and $\phi: E_1 \to E_2$ is a $2$-isogeny over $\mathbb{F}_p$ with $\operatorname{End}_{\mathbb{F}_p}(E_1) \cong \operatorname{End}_{\mathbb{F}_p}(E_2) \cong \mathbb{Z}[\sqrt{-p}]$, then the isogeny $\phi$ corresponds to the ideal $[2, \sqrt{-p}-1]$. Because the ideals $[2, \sqrt{-p}-1]$ and $[\frac{p+1}{2}, \sqrt{-p}-1]$ are equivalent in $C(\mathbb{Z}[\sqrt{-p}])$, there exists an isogeny $\psi:E_1 \to E_2$ defined over $\mathbb{F}_p$ with kernel $\text{Ker}(\psi):=\{ P \in E_1[\frac{p+1}{2}] \mid [\sqrt{-p}-1]P = \infty \}$. We also have $\text{Ker}(\psi)=\{ P \in E_1[\frac{p+1}{2}] \mid [2\sqrt{-p}-2]P = \infty \}$ since $\frac{1+p}{2}$ is odd. Moreover, the kernel ideal of $\psi$ is $[\frac{p+1}{2}, 2\sqrt{-p}-2]$, which corresponds to the form $(\frac{p+1}{2},4, 8)\sim (8 ,-4, \frac{p+1}{2})$.

\begin{proposition}
Let $p \equiv 1 \pmod 4$ and $O=\mathbb{Z}[2\sqrt{-p}]$. The ideal $\mathfrak{l}=[8, 2+2 \sqrt{-p}]$ is proper in $O$. Put $\mathfrak{q}_1=[q_1, 2r_1+2\sqrt{-p}]$. We have an isomorphism: $\mathfrak{l}^{-1} \mathcal{O}(q_1,r_1) \mathfrak{l} \cong \mathcal{O}(q_2,r_2)$ for some $q_2$ satisfying (\ref{e1}) which is the norm of a prime ideal $\mathfrak{q}_2 \in \bar{\mathfrak{q}}_1 \mathfrak{l}^{2} P_{\mathbb{Z}}(2)$.
\end{proposition}

\begin{proof}
Similar to the proof of Proposition \ref{zt3}.

\end{proof}

If $\ell$ is an odd prime, then $\mathfrak{l}=[\ell, -b+2\sqrt{-p}]$ with $b^2 \equiv -4p \pmod {\ell}$. If $\ell=2$ and $p \equiv 1 \pmod 4$, then $\mathfrak{l}=[8,2+2\sqrt{-p}]$. In this case, $\bar{\mathfrak{l}}=[8,-2+2\sqrt{-p}]$ and $\mathfrak{l}^2=\bar{\mathfrak{l}}^2=[4,2\sqrt{-p}]$. We define the following quadratic form $l$.
\begin{equation}\label{5}
l=\left \{
  \begin{array}{ll}
  (8 ,-4, \frac{p+1}{2}) & \text{if} \ \ell=2 ; \\
  (\ell, 2b,\frac{b^2+4p}{\ell}) & \text{if} \ \ell \  \text{is odd .}
  \end{array} \right.
\end{equation}

\begin{theorem}\label{zt5}
Let $E_1$ be a supersingular elliptic curve defined over $\mathbb{F}_p$ with $\operatorname{End} (E_1) \cong  \mathcal{O}(q_1,r_1)$. If $\phi : E_1 \to E_2$ is an isogeny defined over $\mathbb{F}_p$ of degree $\ell$ with $\operatorname{End}_{\mathbb{F}_p}(E_2) \cong \mathbb{Z}[\sqrt{-p}]$, then there is a prime $q_2$ satisfying $(\ref{e1})$ which can be represented by the form $(q_1,4r_1, \frac{4r_1^{2}+4p}{q_1})l^2 $ such that $\operatorname{End}(E_2) \cong \mathcal{O}(q_2,r_2)$, where $l$ is defined in $(\ref{5})$.
\end{theorem}

\begin{proof}
Similar to the proof of Theorem \ref{zt4}.

\end{proof}

\textbf{Case 3:} $\operatorname{End}_{\mathbb{F}_p}( E_1) \ncong \operatorname{End}_{\mathbb{F}_p}( E_2)$.
\\

Considering the structure of isogeny graphs of supersingular elliptic curves over $\mathbb{F}_p$, we can restrict to $\ell=2$ and $p \equiv 3 \pmod 4$.

Let $p \equiv 3 \pmod 4$ be a prime and $q$ a prime satisfying (\ref{e1}). As we know, $r^2+p \equiv 0 \pmod {4q}$ implies $r^2+p \equiv 0 \pmod {q}$. On the contrary, if $r$ satisfies the equation $x^2+p \equiv 0 \pmod {q}$, then $r$ or $r+q$ satisfies the equation $x^2+p \equiv 0 \pmod {4q}$. It follows that $\mathcal{O}(q,r)$ and $\mathcal{O}'(q,r)$ are maximal orders in $B_{p,\infty}$ for the same $q$ and $r$. The following theorem tells us the relations between $\mathcal{O}(q,r)$ and $\mathcal{O}'(q,r)$.

\begin{theorem}\label{zt6}
Let $p \equiv 3 \pmod 4$ be a prime and $E_2$ a supersingular elliptic curve over $\mathbb{F}_p$ with $j(E_2) \neq 1728$. If the endomorphism ring of $E_2$ is isomorphic to $\mathcal{O}(q,r)$, then there exists a $2$-isogeny $\phi : E_1 \to E_2$ such that the endomorphism ring of $E_1$ is $\mathcal{O}'(q,r)$.
\end{theorem}

\begin{proof}
If $j(E_2) \neq 1728$, for the same prime $q$, we consider the order
$$\tilde{\mathcal{O}}(q,r)=\mathbb{Z}+\mathbb{Z}\beta+\mathbb{Z}\frac{1+\alpha+\beta+\alpha\beta}{2} +\mathbb{Z}\frac{(r+\alpha)\beta}{q}.$$
It is easy to show $\tilde{\mathcal{O}}(q,r) \subseteq \mathcal{O}'(q,r) \cap \mathcal{O}(q,r)$, and the discriminant of $\tilde{\mathcal{O}}(q,r)$ is $4p^2$. Since $j(E_2) \neq 1728$, we have $\mathcal{O}'(q,r') \ncong \mathcal{O}(q,r)$ by Proposition \ref{p1}. It follows that $[\mathcal{O}'(q,r') : \tilde{\mathcal{O}}(q)]=[\mathcal{O}(q,r) : \tilde{\mathcal{O}}(q)] =2$ and $\tilde{\mathcal{O}}(q,r) = \mathcal{O}'(q,r) \cap \mathcal{O}(q,r)$. So there exists a $2$-isogeny $\phi : E_1 \to E_2$ with $\operatorname{End}(E_1) \cong \mathcal{O}'(q,r)$.

\end{proof}

\begin{remark}
If $j(E_2) = 1728$, we can assume $E_2:y^2=x^3+4x$. There exists a $2$-isogeny $\phi:E_1 \to E_2$ with $E_1:y^2=x^3-x$ and $\phi(x,y)=(\frac{x^2-1}{x}, \frac{x^2y+y}{x^2})$. Moreover, we have $j(E_1)=1728$ and $\operatorname{End}(E_1) \cong \operatorname{End}(E_2)$.
\end{remark}

In fact, the curve $E_1$ is on the surface and $E_2$ is on the floor. If $p \equiv 7 \pmod 8$, there is only one down $2$-isogeny from $E_1$. If $E_1$ corresponds to the form $(q,r,\frac{r^2+p}{4q})$ with $q$ satisfying $(\ref{e1})$, then $E_2$ corresponds to the form $(q,4r,\frac{4r^2+4p}{q})$.

If $p \equiv 3 \pmod 8$, then there are three down $2$-isogenies from $E_1$. Denote them by $E_{2,i}$ with $i \in \{1,2,3\}$. If $E_1$ corresponds to the form $(q,r,\frac{r^2+p}{4q})$, by Proposition \ref{p2}, there are three forms with discriminant $-16p$ in the non-principal genus that can be derived from $(q,r,\frac{r^2+p}{4q})$. Moreover, each form corresponds to one of $\{E_{2,i}\}_{i \in \{1,2,3\}}$. On the contrary, suppose that the curve $E_{2,i}$ corresponds to the form $(q_i,4r_i,\frac{4r_i^2+4p}{q_i})$ with $q_i$ satisfying $(\ref{e1})$. For every $i$, the form $(q_i,4r_i,\frac{4r_i^2+4p}{q_i})$ can be derived from $(q_i,r_i,\frac{r_i^2+p}{4q_i})$. It follows that the three forms $(q_i,r_i,\frac{r_i^2+p}{4q_i}),i \in \{1,2,3\}$ are equivalent in $C(-p)$ and the curve $E_1$ corresponds to every one of these three forms.

We use the curves in Example \ref{ex1} to illustrate results in this section.

\begin{example}
The notations are as in Example \ref{ex1}. The followings are $2$-isogeny graphs of supersingular elliptic curves over $\mathbb{F}_{83}$.

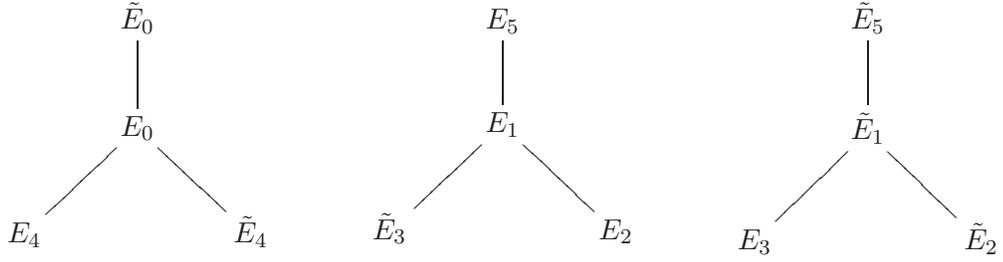
\begin{center}
$\xymatrix{
& &\tilde{E}_0\ar@{-}[d] &  \\
& &E_0 \ar@{-}[ld] \ar@{-}[rd] &  \\
&E_4 && \tilde{E}_4
}$
$\xymatrix{
& &E_5\ar@{-}[d] & \\
& &E_1 \ar@{-}[ld] \ar@{-}[rd] &\\
&\tilde{E}_3 &&E_2
}$
$\xymatrix{
& &\tilde{E}_5\ar@{-}[d] & \\
& &\tilde{E}_1 \ar@{-}[ld] \ar@{-}[rd] &\\
&E_3 &&\tilde{E}_2
}$
\end{center}

Considering the corresponding forms, we have the following graphs.

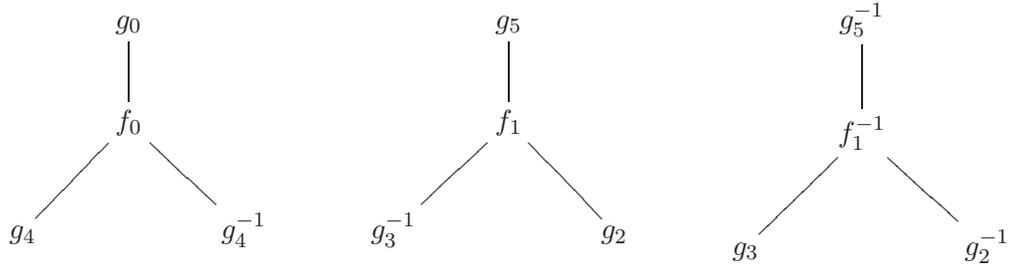
\begin{center}
$\xymatrix{
& &g_0\ar@{-}[d] &  \\
& &f_0 \ar@{-}[ld] \ar@{-}[rd] &  \\
&g_4 && g_4^{-1}
}$
$\xymatrix{
& &g_5\ar@{-}[d] & \\
& &f_1 \ar@{-}[ld] \ar@{-}[rd] &\\
&g_3^{-1} &&g_2
}$
$\xymatrix{
& &g_5^{-1} \ar@{-}[d] & \\
& &f_1^{-1} \ar@{-}[ld] \ar@{-}[rd] &\\
&g_3 &&g_2^{-1}
}$
\end{center}

Note that the forms $g_0$, $g_4$ and $g_4^{-1}$ are derived from $f_0$, and the other two graphs are similar. These examples illustrate Theorem \ref{zt6}. In order to illustrate Theorems \ref{zt4} and \ref{zt5}, we compute the $3$-isogeny graphs of supersingular elliptic curves over $\mathbb{F}_{83}$.
$$
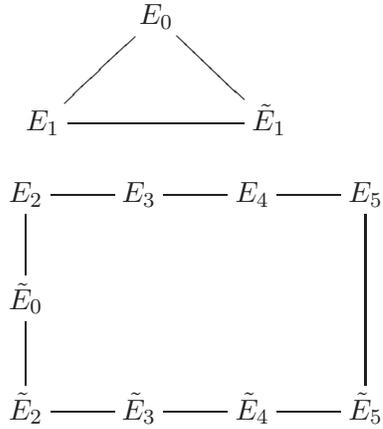
\xymatrix{
& E_0 \ar@{-}[ld] \ar@{-}[rd] & \\
E_1 \ar@{-}[rr] & & \tilde{E}_1
}$$
$$\xymatrix{
& E_2 \ar@{-}[d] \ar@{-}[r]  &E_3 \ar@{-}[r] &E_4 \ar@{-}[r] &E_5 \ar@{-}[dd]  \\
& \tilde{E}_0 \ar@{-}[d] & & & \\
& \tilde{E}_2 \ar@{-}[r]  &\tilde{E}_3 \ar@{-}[r] &\tilde{E}_4 \ar@{-}[r] &\tilde{E}_5
}$$
By considering the corresponding forms, we have the following graphs.
$$
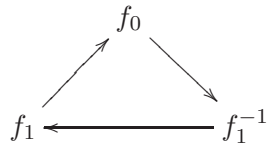
\xymatrix{
& f_0 \ar@{<-}[ld] \ar@{->}[rd] & \\
f_1 \ar@{<-}[rr] & & f_1^{-1}
}$$
Every arrow represents the composition with the form $(3,1,7)^2$.

$$
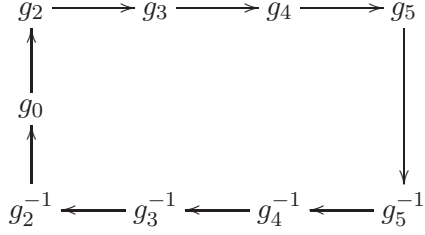
\xymatrix{
& g_2 \ar@{<-}[d] \ar@{->}[r]  &g_3 \ar@{->}[r] &g_4 \ar@{->}[r] &g_5 \ar@{->}[dd]  \\
& g_0 \ar@{<-}[d] & & & \\
& g_2^{-1} \ar@{<-}[r]  &g_3^{-1} \ar@{<-}[r] &g_4^{-1} \ar@{<-}[r] &g_5^{-1}
}$$
Every arrow represents the composition with the form $(3,2,111)^2$.
\end{example}

\section{Applications in CSIDH}

In this section, we recall the encipherment of the CSIDH cryptosystem, and apply our results to analyze the security of it.

Let $k\in \mathbb{Z}_{+}$ and consider a prime $p$ of the form $p=4\ell_1\ell_2 \ldots \ell_k-1$, where the $\ell_i$'s are distinct odd prime numbers. This implies $p \equiv 3 \pmod 8$. CSIDH chooses the curve in the set Iso$(\mathbb{F}_p,\pi)$, which contains all the $\mathbb{F}_p$-isomorphism classes of supersingular elliptic curves with endomorphism ring $\mathcal{O}(q,r)$.

The starting curve of CSIDH is $E_0: y^2=x^3+x$, which corresponds to the form $(4,0,p)$. Alice (resp. Bob) secretly sample $[\mathfrak{a}]$ (resp. $[\mathfrak{b}]$) from the class group of $\mathbb{Z}[\sqrt{-p}]$, compute $E_A=[\mathfrak{a}]E_0$ (resp. $E_B=[\mathfrak{b}]E_0)$, and publish the result. The shared secret is then $E_{AB}=[\mathfrak{a}][\mathfrak{b}]E_0$, which Alice computes as $[\mathfrak{a}]([\mathfrak{b}]E_0)$ and Bob computes as $[\mathfrak{b}]([\mathfrak{a}]E_0)$.

We can write the ideal $[\mathfrak{a}]$ (resp. $[\mathfrak{b}]$) as a quadratic form $l_A$ (resp, $l_B$), and the curve $E_A$ (resp. $E_B$) corresponds to the form $f_A=(4,0,p)l^2_A$ (resp. $f_B=(4,0,p)l^2_B$). Since $f_A f_B=(4,0,p)^2l^2_A l^2_B=l^2_A l^2_B$, the shared secret curve $E_{AB}$ corresponds to the form $f_{AB}=(4,0,p)l^2_A l^2_B=(4,0,p)f_A f_B$.
In order to compute the shared secret of CSIDH, we need to compute the quadratic forms $f_A$ and $f_B$ from the curve $E_A$ and $E_B$ and recover the secret curve $E_{AB}$ from the form $f_{AB}$.

Given a supersingular elliptic curve $E_A$ over $\mathbb{F}_p$, if we know its endomorphism ring $\mathcal{O}(q_1,r_1)$, then it corresponds to the form $(q_1,4r_1,\frac{4r_1^2+4p}{q_1})$ or $(q_1,-4r_1,\frac{4r_1^2+4p}{q_1})$ by Theorem \ref{zt2}. Similarly, the curve $E_B$ corresponds to the form $(q_2,4r_2,\frac{4r_2^2+4p}{q_2})$ or $(q_2,-4r_2,\frac{4r_2^2+4p}{q_2})$. In general, we compute 4 forms and one of them corresponds to the curve $E_{AB}$.

Given a reduced quadratic form, one can compute the endomorphism ring $\mathcal{O}(q,r)$ in polynomial time in length $p$ under GRH (see \cite[Proposition 3.4]{Wesolowski}). Moreover, one can compute the corresponding $j$-invariant by KLPT algorithm in polynomial time \cite{KLPT}.

As we discussed above, if one can compute the binary quadratic forms which correspond to $E_A$ and $E_B$, then it takes polynomial time in length $p$ to get the shared secret of CSIDH. Compared with the method in \cite{MR4095284}, we do not need to solve the discrete logarithm problem in the class group of $\mathbb{Z}[\sqrt{-4p}]$, which is in general quite difficult.

\section{Conclusion}
For a supersingular elliptic curve defined over $\mathbb{F}_p$, Ibukiyama's results show that its endomorphism ring is isomorphic to one of the maximal orders $\mathcal{O}(q,r)$'s and $\mathcal{O}'(q,r')$'s. We show that a quadratic form in the maximal order plays a crucial role, and the actions of $\mathbb{F}_p$-isogenies on elliptic curves is compatible with the composition of quadratic forms. The advantage of this correspondence is that the composition of quadratic forms is commutative, while the multiplication in maximal orders is not. Computing this correspondence explicitly and applying these results to isogeny-based cryptography are our future work.

\section*{Acknowledgement}

The work is supported by the National Key Research and Development Program of China under Grant No. 2022YFA1004900, the National Natural Science Foundation of China under Grant No. 12201637  and No. 62202475, the Natural Science Foundation of Hunan Province of China under Grant No.2021JJ40701, and the Innovation Program for Quantum Science and Technology under Grant No. 2021ZD0302902.

\end{document}